\newfont{\bcb}{msbm10}
\newfont{\matb}{cmbx10}
\newfont{\got}{eufm10}

\documentclass[12pt]{amsart}
\usepackage{amsmath, amsthm, amscd, amsfonts, amssymb, latexsym, graphicx, color}
\usepackage[bookmarksnumbered, colorlinks, plainpages, hypertex]{hyperref}

\usepackage[cp1250]{inputenc}

\usepackage{amsmath,amsthm}
\usepackage{amssymb,latexsym}
\usepackage{enumerate}

\newtheorem{theorem}{Theorem}[section]
\newtheorem{lemma}[theorem]{Lemma}
\newtheorem{proposition}[theorem]{Proposition}
\newtheorem{corollary}[theorem]{Corollary}
\theoremstyle{definition}

\theoremstyle{remark}
\newtheorem{remark}[theorem]{Remark}
\numberwithin{equation}{section}

\begin{document}

\title[Geometry over fields with analytic structure]{Some results
       of geometry over \\
       Henselian fields with analytic structure}

\author[Krzysztof Jan Nowak]{Krzysztof Jan Nowak}




\date{}

\begin{abstract}
The paper deals with Henselian valued fields (of
equi\-characteristic zero) with analytic structure. Actually, we
are focused on separated analytic structures, but the results
remain valid for strictly convergent analytic ones as well. A
classical example of the latter is a complete rank one valued
fields with the Tate algebra of strictly convergent power series.
The algebraic case was treated in our previous papers. Here we are
going to carry over the research to the general analytic settings.
Also considered are the local rings of analytic function germs
induced by a given separated Weierstrass system, which turn out to
be excellent and regular. Several results are established as, for
instance, piecewise continuity of definable functions, several
versions of the \L{}ojasiewicz inequality, H\"{o}lder continuity
of definable functions continuous on closed bounded subsets of the
affine space and curve selection for definable sets. Likewise as
before, at the center of our approach is the closedness theorem to
the effect that every projection with closed bounded fiber is a
definably closed map. It enables application of resolution of
singularities and of transformation to a normal crossing by
blowing up (here applied to the local rings of analytic function
germs) in much the same way as over locally compact ground fields.
We rely on elimination of valued field quantifiers, term structure
of definable functions and b-minimal cell decomposition, due to
Cluckers--Lipshitz--Robinson, as well as on relative quantifier
elimination for ordered abelian groups, due to
Cluckers--Halupczok. Besides, other two ingredients of the proof
of the closedness theorem are existence of the limit (after finite
partitioning of the domain) for a definable function of one
variable and fiber shrinking, being a relaxed version of curve
selection.
\end{abstract}

\maketitle

\section{Introduction}

The paper deals with Henselian valued fields of equicharacteristic
zero with analytic structure. We are focused on separated analytic
structures, whose theory is briefly recalled in Section~2.
However, the results established here remain valid in that of
strictly convergent analytic structures, because every such a
structure can be extended in a definitional way (extension by
Henselian functions) to a separated analytic structure
(cf.~\cite{C-Lip}). Complete, rank one valued fields with the Tate
algebra of strictly convergent power series are a classical
example. Geometry over Henselian valued fields in the algebraic
case was treated in our previous articles~\cite{Now-Sel,Now-thm}.
We are now going to carry over the research to the general
analytic settings.

\vspace{1ex}

Throughout the paper, we shall usually assume that the ground
valued field $K$ is of equicharacteristic zero, not necessarily
algebraically closed. Denote by $v$, $\Gamma = \Gamma_{K}$,
$K^{\circ}$, $K^{\circ \circ}$ and $\widetilde{K}$ the valuation,
its value group, the valuation ring, maximal ideal and residue
field, respectively. The multiplicative norm corresponding to $v$
will be denoted by $| \cdot |$. By the $K$-topology on $K^{n}$ we
mean the topology induced by the valuation $v$. As before, at the
center of our approach is the following closedness theorem

\begin{theorem}\label{clo-th}
Let $K$ be a Henselian valued field with separated analytic
structure in the analytic language $\mathcal{L}$. Given an
$\mathcal{L}$-definable subset $D$ of $K^{n}$, the canonical
projection
$$ \pi: D \times (K^{\circ})^{m} \longrightarrow D  $$
is definably closed in the $K$-topology, i.e.\ if $B \subset D
\times (K^{\circ})^{m}$ is an $\mathcal{L}$-definable closed
subset, so is its image $\pi(B) \subset D$.
\end{theorem}

\vspace{1ex}

We immediately obtain two consequences.

\begin{corollary}\label{clo-th-cor-0}
Let $D$ be an $\mathcal{L}$-definable subset of $K^{n}$ and
$\,\mathbb{P}^{m}(K)$ stand for the projective space of dimension
$m$ over $K$. Then the canonical projection
$$ \pi: D \times \mathbb{P}^{m}(K) \longrightarrow D $$
is definably closed. \hspace*{\fill} $\Box$
\end{corollary}

\begin{corollary}\label{clo-th-cor-1}
Let $A$ be a closed $\mathcal{L}$-definable subset of
$\,\mathbb{P}^{m}(K)$ or of $(K^{\circ})^{m}$. Then every
continuous $\mathcal{L}$-definable map $f: A \to K^{n}$ is
definably closed in the $K$-topology. \hspace*{\fill} $\Box$
\end{corollary}

Theorem~\ref{clo-th} will be proven in Section~4. The strategy of
proof in the analytic settings will generally follow the one in
the algebraic case from our papers~\cite{Now-Sel,Now-thm}. Here we
apply elimination of valued field quantifiers for the theory
$T_{Hen,\mathcal{A}}$ along with b-minimal cell decompositions
with centers (Theorem~\ref{ball}) and term structure of definable
functions (Theorem~\ref{term}), both the results due to
Cluckers--Lipshitz--Robinson~\cite{C-Lip-R,C-Lip-0}, as well as
relative quantifier elimination for ordered abelian groups (in a
many-sorted language with imaginary auxiliary sorts), due to
Cluckers--Halupczok~\cite{C-H}. Besides, in the proof of the
closedness theorem, we rely on the local behavior of definable
functions of one variable and on fiber shrinking, being a relaxed
version of curve selection.

\vspace{1ex}

Majority of the results in the subsequent sections rely on the
closedness theorem. It enables, in particular, application of
resolution of singularities and of transformation to a normal
crossing by blowing up in much the same way as over locally
compact ground fields.

\vspace{1ex}

Note that non-Archimedean analytic geometry over Henselian valued
fields has a long history (see e.g.\
\cite{De-Dries,Dries,Lip,Dries-Mac,Dries-Has,L-R-0,L-R,C-Lip-R,C-Lip-0,C-Lip}).
The concept of fields with analytic structure, introduced by
Cluckers--Lipshitz--Robinson~\cite{C-Lip-R}, will be recalled for
the reader's convenience in Section~2 following the last two
papers listed above. In Section~3, we consider the local rings of
analytic functions germs $A_{p}^{loc}(K)$ at $0 \in K^{p}$, $p \in
\mathbb{N}$, induced by a given separated Weierstrass system. They
turn out to be excellent regular local rings.

\vspace{1ex}

In Section~5, we give two direct applications of the closedness
theorem, namely theorems on existence of the limit
(Proposition~\ref{limit-th2}) and on piecewise continuity
(Theorem~\ref{piece}). Note that our proof of the closedness
theorem makes use of a certain version of the former result
(Proposition~\ref{limit-th1}).

\vspace{1ex}

Section~6 contains several versions of the \L{}ojasiewicz
inequality with an immediate consequence, H\"{o}lder continuity of
definable functions continuous on closed bounded subsets of
$K^{n}$. We only state the results, because the proofs of the
algebraic versions from our papers~\cite{Now-Sel,Now-thm} can be
repeated almost verbatim.

\vspace{1ex}

Finally, in the last section, we establish a general version of
curve selection for definable sets. It differs from the classical
one in that the domain of the selected curve is only a definable
subset of the unit disk. Its proof relies on the closedness
theorem, transformation to a normal crossing by blowing up
(applied to excellent regular local rings $A_{p}^{loc}(K)$ of
analytic function germs examined in Section~3), elimination of
valued field quantifiers, relative quantifier elimination for
ordered abelian groups, and a result from piecewise linear
geometry to which fiber shrinking comes down.



\section{Fields with analytic structure}
In this section, we remind the reader, following the
paper~\cite{C-Lip-0}, the concept of a separated analytic
structure.

\vspace{1ex}

Let $A$ be a commutative ring with unit and with a fixed proper
ideal $I \varsubsetneq A$; put $\widetilde{A} = A/I$. A {\em
separated} $(A, I)$-{\em system} is a certain system $\mathcal{A}$
of $A$-subalgebras $A_{m,n} \subset A[[\xi,\rho]]$, $m,n \in
\mathbb{N}$; here $A_{0,0} = A$ (op.~cit., Section~4.1). Two kinds
of variables, $\xi = (\xi_{1}, \ldots,\xi_{m})$ and $\rho =
(\rho_{1}, \ldots, \rho_{n})$, play different roles. Roughly
speaking, the variables $\xi$ vary over the valuation ring (or the
closed unit disc) $K^{\circ}$ of a valued field $K$, and the
variables $\rho$ vary over the maximal ideal (or the open unit
disc) $K^{\circ \circ}$ of $K$.

\vspace{1ex}

For a power series $f \in A[[ \xi, \rho ]]$, we say that

\vspace{1ex}

1) $f$ is $\xi_{m}$-{\em regular of degree} $d$ if $f$ is
congruent to a monic polynomial in $\xi_{m}$ of degree $d$ modulo
the ideal
$$ I[[ \xi,\rho ]] + (\rho) A[[ \xi,\rho ]] ; $$

2) $f$ is $\rho_{n}$-{\em regular of degree} $d$ if $f$ is
congruent to $\rho_{n}^{d}$ modulo the ideal
$$ I[[ \xi,\rho ]] + (\rho_{1},\ldots,\rho_{n-1},\rho_{n}^{d+1}) A[[ \xi,\rho ]] . $$

The $(A, I)$-system $\mathcal{A}$ is called a {\em separated
pre-Weierstrass system} if two usual Weierstrass division theorems
hold with respect to division by each $f \in A_{m,n}$ which is
$\xi_{m}$-regular or $\rho_{n}$-regular. By Weierstrass division,
units of $A_{m,n}$, being power series regular of degree $0$, are
precisely elements of the form $c + g$, where $c$ is a unit of $A$
and
$$ g \in A_{m,n}^{\circ} := (I,\rho) \, A_{m,n}. $$

\vspace{1ex}

Also introduced is the concept of rings $C$ of $\mathcal{A}$-{\em
fractions} with proper ideal $C^{\circ}$ (where $C^{\circ} := I$
if $C=A$) and with rings $C_{mn}$ of separated power series over
$C$; put $C_{m,n}^{\circ} := (C^{\circ},\rho)C_{m,n}$.

\vspace{1ex}

A pre-Weierstrass system $\mathcal{A}$ is called a {\em separated
Weierstrass system} if the rings $C$ of fractions enjoy the
following weak Noetherian property: If
$$ f = \sum_{\mu,\nu} \, c_{\mu,\nu} \xi^{\mu}\rho^{\nu} \in
   C_{m,n} \ \ \text{with} \ \ c_{\mu,\nu} \in C, $$
then there exist a finite set $J \subset \mathbb{N}^{m+n}$ and
elements $g_{\mu,\nu} \in C_{m,n}^{\circ}$, \ $(\mu,\nu) \in J$,
such that
$$ f = \sum_{(\mu,\nu) \in J} \, c_{\mu,\nu}\,
   \xi^{\mu} \, \rho^{\nu} \, (1+ g_{\mu,\nu}). $$
The above condition is a form of Noetherianity and implies, in
particular, that if
$$ f = \sum_{\mu,\nu} \, a_{\mu \nu} \, \xi^{\mu} \rho^{\nu} \in
   A_{m,n}, $$
then all the coefficients $a_{\mu \nu}$ are linear combinations of
finitely many of them and, moreover, if a coefficient $a_{\mu
\nu}$ is "small", so can be the coefficients of such a
combination. Consequently, the Gauss norm on each $A_{m,n}$ is
defined whenever $A = F^{\circ}$ and $I = F^{\circ \circ}$ for a
valued field $F$. Moreover, then the weak Noetherian property is
equivalent to the condition that for every $f \in A_{m,n}$, $f
\neq 0$, there is an element $c \in F$ such that $cf \in A_{m,n}$
and $\| cf \| = 1$.

\vspace{1ex}

Let $\mathcal{A}$ be a separated Weierstrass system and $K$ a
valued field. A {\em separated analytic} $\mathcal{A}$-{\em
structure} on $K$ is a collection of homomorphisms $\sigma_{m,n}$
from $A_{m,n}$ to the ring of $K^{\circ}$-valued functions on
$(K^{\circ})^{m} \times (K^{\circ \circ})^{n}$, $m,n \in
\mathbb{N}$, such that

1) $\sigma_{0,0} (I) \subset K^{\circ \circ}$;

2) $\sigma_{m,n}(\xi_{i})$ and $\sigma_{m,n}(\rho_{j})$ are the
$i$-th and $(m+j)$-th coordinate functions on $(K^{\circ})^{m}
\times (K^{\circ \circ})^{n}$, respectively;

3) $\sigma_{m+1,n}$ and $\sigma_{m,n+1}$ extend $\sigma_{m,n}$,
where functions on $(K^{\circ})^{m} \times (K^{\circ \circ})^{n}$
are identified with those functions on
$$ (K^{\circ})^{m+1} \times (K^{\circ \circ})^{n} \ \ \ \text{or} \
   \ \ (K^{\circ})^{m} \times (K^{\circ \circ})^{n+1} $$
which do not depend on the coordinate $\xi_{m+1}$ or $\rho_{n+1}$,
respectively.

\vspace{1ex}

It can be shown via Weierstrass division that analytic
$\mathcal{A}$-structures preserve composition; more precisely,
functions from $A_{k,l}$ may substitute for the variables $\xi$
and functions from $A_{k,l}^{\circ}$ may substitute for the
variables $\rho$ (op.~cit., Proposition~4.5.3). When considering a
particular field $K$ with analytic $\mathcal{A}$-structure, one
may assume that $\mathrm{ker}\, \sigma_{0,0} = (0)$. Indeed,
replacing $A$ by $A/\mathrm{ker}\, \sigma_{0,0}$ yields an
equivalent analytic structure on $K$ with this property. Then $A =
A_{0,0}$ can be regarded as a subring of $K^{\circ}$.

\vspace{1ex}

If the ground field $K$ is trivially valued, then $K^{\circ \circ}
= (0)$ and the analytic structure reduces to the algebraic
structure given by polynomials. If $K$ is non-trivially valued,
then the function induced by a power series from $A_{m,n}$, $m,n
\in \mathbb{N}$, is the zero function iff the image in $K$ of each
of its coefficients is zero (op.~cit., Proposition~4.5.4). A
separated analytic $\mathcal{A}$-structure on a valued field $K$
can be uniquely extended to any algebraic extension $K'$ of $K$;
in particular, to the algebraic closure $K_{alg}$ of $K$
(op.~cit., Theorem~4.5.11). (These properties remain valid for
strictly convergent analytic structures as well.) Every valued
field with separated analytic structure is Henselian (op.~cit.,
Proposition~4.5.10).

\vspace{1ex}

\begin{remark}\label{ext-par}
From now on, we shall always assume that the ground field $K$ is
non-trivially valued and that $\sigma_{0,0}$ is injective. Under
the assumptions, one can canonically obtain by extension of
parameters a (unique) separated Weierstrass system
$\mathcal{A}(K)$ over $(K^{\circ},K^{\circ \circ})$ so that $K$
has separated analytic $\mathcal{A}(K)$-structure; a similar
extension can be performed for any subfield $F \subset K$
(op.~cit., Theorem~4.5.7). (This technique holds for strictly
convergent Weierstrass systems as well.)
\end{remark}

\vspace{1ex}

Now we can describe the analytic language $\mathcal{L}$ of an
analytic structure $K$ determined by a separated Weierstrass
system $\mathcal{A}$. We begin by defining the semialgebraic
language $\mathcal{L}_{Hen}$. It is a two sorted language with the
main, valued field sort $K$, and the auxiliary $RV$-sort
$$ RV = RV(K) := RV^{*}(K) \cup \{ 0 \}, \ \ \
   RV^{*}(K) := K^{\times}/(1 + K^{\circ \circ}); $$
here $A^{\times}$ denotes the set of units of a ring $A$. The
language of the valued field sort is the language of rings
$(0,1,+,-, \cdot)$. The language of the auxiliary sort is the
so-called inclusion language (op.~cit., Section~6.1). The only map
connecting the sorts is the canonical map
$$ rv: K \to RV(K), \ \ \ 0 \mapsto 0. $$
Since
$$ \widetilde{K}^{\times} \simeq (K^{\circ})^{\times}/(1 + K^{\circ \circ})
   \ \ \ \text{and} \ \ \ \Gamma \simeq K^{\times}/(K^{\circ})^{\times}, $$
we get the canonical exact sequence
$$ 1 \to \widetilde{K}^{\times} \to RV^{*}(K) \to \Gamma \to 0. $$
This sequence splits iff the valued field $K$ has an angular
component map.

\vspace{1ex}

The analytic language $\mathcal{L} =
\mathcal{L}_{Hen,\mathcal{A}}$ is the semialgebraic language
$\mathcal{L}_{Hen}$ augmented on the valued field sort $K$ by the
reciprocal function $1/x$ (with $1/0 :=0$) and the names of all
functions of the system $\mathcal{A}$, together with the induced
language on the auxiliary sort $RV$ (op.~cit., Section~6.2). A
power series $f \in A_{m,n}$ is construed via the analytic
$\mathcal{A}$-structure on their natural domains and as zero
outside them. More precisely, $f$ is interpreted as a function
$$ \sigma (f) = f^{\sigma}: (K^{\circ})^{m} \times (K^{\circ \circ})^{n} \to
   K^{\circ}, $$
extended by zero on $K^{m+n} \setminus (K^{\circ})^{m} \times
(K^{\circ \circ})^{n}$.

\vspace{1ex}

In the equicharacteristic case, however, the induced language on
the auxiliary sort $RV$ coincides with the semialgebraic inclusion
language. It is so because then \cite[Lemma~6.3.12]{C-Lip-0} can
be strengthen as follows, whereby \cite[Lemma~6.3.14]{C-Lip-0} can
be directly reduced to its algebraic analogue. Consider a strong
unit on the open ball $B = K_{alg}^{\circ \circ}$. Then $rv
(E^{\sigma})(x)$ is constant when $x$ varies over $B$. This is no
longer true in the mixed characteristic case. There, a weaker
conclusion asserts that the functions $rv_{n} (E^{\sigma})(x)$, $n
\in \mathbb{N}$, depend only on $rv_{n}(x)$ when $x$ varies over
$B$; actually, $rv_{n} (E^{\sigma})(x)$ depend only on $\: x \! \!
\mod (n \cdot K^{\circ \circ}_{alg})$ when $x$ varies over $B$, as
indicated in~\cite[Remark~A.1.12]{C-Lip}. Under the circumstances,
the residue field $\widetilde{K}$ is orthogonal to the value group
$\Gamma_{K}$, whenever the ground field $K$ has an angular
component map or, equivalently, when the auxiliary sort $RV$
splits (in a non-canonical way):
$$ RV^{*}(K) \simeq \widetilde{K}^{\times} \times \Gamma_{K}. $$
This means that every definable set in $\widetilde{K}^{m} \times
\Gamma_{K}^{n}$ is a finite union of the Cartesian products of
some sets definable in $\widetilde{K}^{m}$ (in the language of
rings) and in $\Gamma_{K}^{n}$ (in the language of ordered
groups). The orthogonality property will often be used in the
paper, similarly as it was in the algebraic case treated in our
papers~\cite{Now-Sel,Now-thm}.

\begin{remark}\label{angular}
Not all valued fields $K$ have an angular component map, but it
exists if $K$ has a cross section, which happens whenever $K$ is
$\aleph_{1}$-saturated (cf.~\cite[Chap.~II]{Ch}). Moreover, a
valued field $K$ has an angular component map whenever its residue
field $\Bbbk$ is $\aleph_{1}$-saturated
(cf.~\cite[Corollary~1.6]{Pa2}). In general, unlike for $p$-adic
fields and their finite extensions, adding an angular component
map does strengthen the family of definable sets. Since the
$K$-topology is $\mathcal{L}$-definable, the closedness theorem is
a first order property. Therefore it can be proven using
elementary extensions, and thus one may assume that an angular
component map exists.
\end{remark}

Denote by $\mathcal{L}^{*}$ the analytic language $\mathcal{L}$
augmented by all Henselian functions
$$ h_{m}: K^{m+1} \times RV(K) \to K,  \ \ m \in \mathbb{N}, $$
which are defined by means of a version of Hensel's lemma
(cf.~\cite{C-Lip-0}, Section~6).

\vspace{1ex}

Let $\mathcal{T}_{Hen,\mathcal{A}}$ be the theory of all Henselian
valued fields of characteristic zero with separated analytic
$\mathcal{A}$-structure. Two crucial results about analytic
structures are Theorems~6.3.7 and~6.3.8 from~\cite{C-Lip-0}),
stated below.

\begin{theorem}\label{ball}
The theory $\mathcal{T}_{Hen,\mathcal{A}}$ eliminates valued field
quantifiers, is b-minimal with centers and preserves all balls.
Moreover, $\mathcal{T}_{Hen,\mathcal{A}}$ has the Jacobian
property.
\end{theorem}

\begin{theorem}\label{term}
Let $K$ be a Henselian field with separated analytic
$\mathcal{A}$-structure. Let $f: X \to K$, $X \subset K^{n}$, be
an $\mathcal{L}(B)$-definable function for some set of parameters
$B$. Then there exist an $\mathcal{L}(B)$-definable function $g: X
\to S$ with $S$ auxiliary and an $\mathcal{L}^{*}(B)$-term $t$
such that
$$ f(x) = t(x,g(x)) \ \ \ \text{for all} \ \ x \in X. $$
\end{theorem}

It follows from Theorem~\ref{ball} that the theory
$\mathcal{T}_{Hen,\mathcal{A}}$ admits b-minimal cell
decompositions with centers (cf.~\cite{C-L}).

\vspace{1ex}

\section{Rings of analytic function germs}

Keeping the notation from~\cite{C-Lip-0}, put
$$ A_{m,n}^{\dag}(K) := K \otimes_{K^{\circ}} A_{m,n}(K). $$
By the weak Noetherian property, we immediately get
$$ A_{m,n}^{\dag}(K)^{\circ} := \{ f \in A_{m,n}^{\dag}(K): \, \|
   f \| \leq 1 \} = A_{m,n}(K) = $$
$$ \{ f \in A_{m,n}^{\dag}(K): \, |f^{\sigma}(a,b)| \leq 1 \ \ \text{for all} \
   (a,b) \in (K_{alg}^{\circ})^{m} \times (K_{alg}^{\circ \circ})^{n} \} $$
and
$$ A_{m,n}^{\dag}(K)^{\circ \circ} := (K^{\circ \circ},\rho) \, A_{m,n}^{\dag}(K)^{\circ}
   = A_{m,n}(K)^{\circ} = $$
$$ \{ f \in A_{m,n}^{\dag}(K): \, |f^{\sigma}(a,b)| < 1 \ \ \text{for all} \
   (a,b) \in (K_{alg}^{\circ})^{m} \times (K_{alg}^{\circ \circ})^{n} \}. $$

\vspace{1ex}

It is demonstrated in~\cite[Section~5]{C-Lip-0} that the classical
R\"{u}ckert theory from \cite[Section~5.2.5]{B-G-R} applies to the
rings $A_{m,0}^{\dag}(K)$ and $A_{0,n}^{\dag}(K)$. Indeed, for any
$f \in A_{m,0}(K)$ or $f \in A_{0,n}(K)$, $f \neq 0$, there is an
$a \in K$ such that $af \in A_{m,0}(K)$ or $af \in A_{0,n}(K)$,
respectively, $\| af \| = 1$, and $af$ is regular after a
Weierstrass change of variables. Hence those rings enjoy many good
algebraic properties. The authors conjecture that those properties
are shared also by the rings $A_{m,n}(K)$ of separated analytic
functions, which seems to be a problem yet unsolved.

We shall demonstrate that the rings of analytic function germs are
excellent and regular. Therefore resolution of singularities can
be applied to them, which will be used in the proof of curve
selection in Section~7. Put
$$ \Delta_{m,n}(r) := \{ (\xi,\rho) \in K^{m+n}: \;
   |\xi_{i} | \leq r , \ |\rho_{j} | < r \} = c \cdot \Delta_{n,m}(1) $$
with $c \in K$, $c \neq 0,$ and $r = |c|$. For any $f \in
A_{m,n}(K)$, the power series
$$ f_{c}(\xi,\rho) := f(\xi/c,\rho/c) \in K[[\xi,\rho]] $$
determines a function
$$ f_{c}^{\sigma} : \Delta_{m,n}(r) \longrightarrow K^{\circ}. $$
Put
$$ A_{m,n}(K,r) := \{ f^{c}: \, f \in A_{m,n}(K) \}, \ \
   A_{m,n}^{\dag}(K,r) := K \otimes_{K^{\circ}} A_{m,n}(K,r), $$
(obviously, this definition does not depend on the choice of $c$
with $|c|=r$), and let
$$ A_{m,n}^{loc}(K) := \bigcup_{r>0}\, A_{m,n}^{r}(K) $$
be the direct limit of the system of rings
$$ A_{m,n}^{\dag}(K,r) \subset A_{m,n}^{\dag}(K,s),
   \ \ 0 < s <r. $$
Under the assumptions imposed on analytic structures throughout
the paper, one can identify a power series $f_{c}$ with the
analytic function $f_{c}^{\sigma}$. It is not difficult to check
that $A_{m,n}^{loc}(K)$ are local rings with maximal ideal
generated by the variables $\xi$ and $\rho$.

\begin{remark}\label{rem-germs}
In fact, it follows from~\cite[Prop.~4.5.3]{C-Lip-0} (on
preserving composition) that the rings $A_{m,n}^{loc}(K)$ depend
only on $p=m+n$. Hence the variables $\xi$ and $\rho$ play locally
the same role, and thus the use of double indices $m,n$ is
immaterial and refers merely to the names of variables. Therefore
it is more natural to use the notation
$$ A_{p}^{loc}(K) := A_{m,n}^{loc}(K), \ \ p=m+n, $$
and call $A_{p}^{loc}(K)$ the ring of analytic function germs at
$0 \in K^{p}$ induced by the separated Weierstrass system.
\end{remark}

\vspace{1ex}

Also note that the Gauss norm and supremum norm on
$A_{m,n}^{\dag}(K)$ coincide whenever the residue field
$\widetilde{K}$ is infinite (maximal modulus principle;
cf.~\cite[Remark~5.2.8]{C-Lip-0} or \cite[Section~5.1.4,
Proposition~3]{B-G-R} for the classical case).

\vspace{1ex}

Consider now any $f \in A_{m,n}(K)$ and $c \in K^{\circ \circ}$,
$c \neq 0$. Similarly as in the classical case of strictly
convergent power series elaborated in~\cite[Section~5.2.3
and~5.2.4]{B-G-R}, the power series $a f_{1/c}$ is, after a
Weierstrass change of variables, regular with respect to any of
the variables $\xi$ or $\rho$ for some $a \in K$. Therefore
Weierstrass preparation and division, as well as the Weierstrass
finiteness theorem hold for the rings $A_{p}^{loc}(K)$. The latter
enables induction on the number of indeterminates. In particular,
the residue field is a finite extension of the ground field $K$.
Hence also applicable here is R\"{u}ckert's theory (op.~cit.,
Section~5.2.5) We can thus state the following

\begin{theorem}
The rings $A_{p}^{loc}(K)$, $p \in \mathbb{N}$, are local
Noetherian factorial Jacobson rings.  \hspace*{\fill} $\Box$
\end{theorem}

Hence we obtain two following

\begin{corollary}
The rings $A_{p}^{loc}(K)$, $p \in \mathbb{N}$, are excellent
regular local ring of dimension $p$.
\end{corollary}

\begin{proof}
The dimension of $A_{p}^{loc}(K) = A_{m,n}^{loc}(K)$ is $\geq
p=m+n$ because one has the chain of prime ideals
$$ (0) \subset (\mathbb{}\xi_{1}) \subset \ldots \subset
   (\xi_{1},\ldots,\xi_{m},\rho_{1}) \subset \ldots \subset
   (\xi_{1},\ldots,\xi_{m},\rho_{1}, \ldots,\rho_{n}). $$
Since the maximal ideal is generated by $p=m+n$ variables $\xi$
and $\rho$, the converse inequality and the regularity of
$A_{p}^{loc}(K)$ follows immediately. Finally, observe that
$A_{p}^{loc}(K)$ contains its residue field $K$ of characteristic
$0$, and that
$$ \partial/\partial \xi_{1}, \ldots, \partial/\partial \xi_{m},
   \partial/\partial \rho_{1}, \ldots, \partial/\partial \rho_{n}  $$
are derivatives of $A_{p}^{loc}(K)$ over $K$ such that
$$ \partial \xi_{i}/\partial \xi_{j} = \delta_{ij} \ \ \text{and} \ \
   \partial \rho_{i}/\partial \rho_{j} = \delta_{ij}. $$
Hence and by the Jacobian criterion for
excellence~\cite[Theorem~102]{Mat}, the ring $A_{p}^{loc}(K)$ is
excellent, concluding the proof.
\end{proof}

\vspace{1ex}

Finally, note that resolution of singularities applies to
excellent regular schemes (cf.~\cite[Theorem~1.1.3]{Tem}
or~\cite{Hi}, Corollary~3 on p.~146 along with the explanations on
p.~161, for the classical case of excellent regular local rings).
In particular, it applies to the scheme $X := \mathrm{Spec}\,
(A_{p}^{loc}(K))$. Below we state a version which refers, in fact,
to transformation to a normal crossing by blowing up.

\begin{theorem}
Let $f_{1},\ldots,f_{k} \in A_{p}^{loc}(K)$. Then there exists a
finite composite of blow-ups $\pi: \widetilde{X} \to X$ along
smooth centers such that $f_{1} \circ \pi, \ldots, f_{k} \circ
\pi$ are simple normal crossing divisors on $\widetilde{X}$.
\hspace*{\fill} $\Box$
\end{theorem}

\begin{remark}
It is well known that in the conclusion one can require
$$ f_{1} \circ \pi, \ldots, f_{k} \circ \pi $$
to be linearly ordered with respect to divisibility relation at
each point in $\widetilde{X}$.
\end{remark}

It is possible to adapt here Serre's concept of analytic manifolds
(\cite{Se}, Part~II, Chap.~III) with respect to the class of
analytic function germs induced (via translations) by
$A_{p}^{loc}(K)$. Though this concept is quite weak, it is
sufficient for our further applications. For simplicity, we shall
denote by the same symbol a function germ or a set germ and its
representative. This convention will not lead to confusion.
Clearly, $X$, $\widetilde{X}$ and the map $\pi: \widetilde{X} \to
X$ correspond respectively to:

1) a polydisk $\Delta_{p}(r) := \Delta_{p,0}(r)$ and a closed (in
the $K$-topology) analytic submanifold $\widetilde{X}_{0}$ of
$\Delta_{p}(r) \times \mathbb{P}^{N}(K)$, for some $0< r <1$ and
$N \in \mathbb{N}$;

2) the restriction to $\widetilde{X}_{0}$ of the projection of
$\Delta_{p}(r) \times \mathbb{P}^{N}(K)$ onto the first factor.

We may regard $X_{0}$ and $\widetilde{X}_{0}$ as the sets of
"$K$-rational points" of $X$ and $\widetilde{X}$, respectively.
Summing up, we obtain

\begin{corollary}\label{cor-NC}
Let $f_{1},\ldots,f_{k} \in A_{p}^{loc}(K)$. Then there exists a
finite composite of blow-ups $\pi_{0}: \widetilde{X}_{0} \to
\Delta_{p}(r)$ along smooth analytic submanifolds such that the
pull-backs $f_{1} \circ \pi_{0}, \ldots, f_{k} \circ \pi_{0}$ are
simple normal crossing divisors on $\widetilde{X}_{0}$. The map
$\pi_{0}$ is surjective and definably closed by virtue of the
closedness theorem.

Moreover, one can ensure that
$$ f_{1} \circ \pi_{0}, \ldots, f_{k} \circ \pi_{0} $$
are linearly ordered with respect to divisibility relation at each
point in $\widetilde{X}_{0}$. It means that at each point $a \in
\widetilde{X}_{0}$ there are local analytic coordinates
$x=(x_{1},\ldots,x_{p})$ with $x(a)=0$ such that
$$ f_{i} \circ \pi_{0} (x) = u_{i}(x) \cdot x^{\alpha_{i}}, \ \
   i=1,\ldots,k, $$
where the $u_{i}$ are analytic units at $a$, $u_{i}(a) \neq 0$,
$\alpha_{i} \in \mathbb{N}^{p}$, and the monomials
$x^{\alpha_{i}}$, $i=1,\ldots,k$, are linearly ordered by
divisibility relation.
\end{corollary}

\section{Proof of the closedness theorem}

In the algebraic case, the proofs of the closedness theorem given
in our papers~\cite{Now-Sel,Now-thm}) make use of the following
three main tools:

$\bullet$ elimination of valued field quantifiers and cell
decomposition due to Pas;

$\bullet$ fiber shrinking
(\cite[Proposition~6.1]{Now-Sel,Now-thm});

$\bullet$ and the theorem on existence of the limit
(\cite[Proposition~5.2]{Now-Sel} and \cite[Theorem~5.1]{Now-thm}).

\vspace{1ex}

In this paper, we apply analytic versions of quantifier
elimination, cell decomposition and term structure
(Theorems~\ref{ball} and~\ref{term}) due to
Cluckers--Lipshitz--Robinson.

\vspace{1ex}

Now recall the concept of fiber shrinking. Let $A$ be an
$\mathcal{L}$-definable subset of $K^{n}$ with accumulation point
$$ a = (a_{1},\ldots,a_{n}) \in K^{n} $$
and $E$ an $\mathcal{L}$-definable subset of $K$ with accumulation
point $a_{1}$. We call an $\mathcal{L}$-definable family of sets
$$ \Phi = \bigcup_{t \in E} \ \{ t \} \times \Phi_{t} \subset A $$
an $\mathcal{L}$-definable $x_{1}$-fiber shrinking for the set $A$
at $a$ if
$$ \lim_{t \rightarrow a_{1}} \, \Phi_{t} = (a_{2},\ldots,a_{n}),
$$
i.e.\ for any neighbourhood $U$ of $(a_{2},\ldots,a_{n}) \in
K^{n-1}$, there is a neighbourhood $V$ of $a_{1} \in K$ such that
$\emptyset \neq \Phi_{t} \subset U$ for every $t \in V \cap E$, $t
\neq a_{1}$. When $n=1$, $A$ is itself a fiber shrinking for the
subset $A$ of $K$ at an accumulation point $a \in K$.


\begin{proposition}\label{FS} (Fiber shrinking)
Every $\mathcal{L}$-definable subset $A$ of $K^{n}$ with
accumulation point $a \in K^{n}$ has, after a permutation of
coordinates, an $\mathcal{L}$-definable $x_{1}$-fiber shrinking at
$a$.
\end{proposition}

Its proof was reduced, by means of elimination of valued field
quantifiers, to Lemma~\ref{line-1} below
(\cite[Lemma~6.2]{Now-thm}), which, in turn, was obtained via
relative quantifier elimination for ordered abelian groups. That
approach can be repeated verbatim in the analytic settings.

\begin{lemma}\label{line-1}
Let $\Gamma$ be an ordered abelian group and $P$ be a definable
subset of $\Gamma^{n}$. Suppose that $(\infty,\ldots,\infty)$ is
an accumulation point of $P$, i.e.\ for any $\delta \in \Gamma$
the set
$$ \{ x \in P: x_{1} > \delta, \ldots, x_{n} > \delta \} \neq \emptyset $$
is non-empty. Then there is an affine semi-line
$$ L = \{ (r_{1}t + \gamma_{1},\ldots,r_{n}t + \gamma_{n}): \, t
   \in \Gamma, \ t \geq 0 \} \ \ \ \text{with} \ \ r_{1},\ldots,r_{n}
   \in \mathbb{N}, $$
passing through a point $\gamma = (\gamma_{1},\ldots,\gamma_{n})
\in P$ and such that $(\infty,\ldots,\infty)$ is an accumulation
point of the intersection $P \cap L$ too. \hspace*{\fill} $\Box$
\end{lemma}

In a similar manner, one can obtain the following

\begin{lemma}\label{line-2}
Let $P$ be a definable subset of $\Gamma^{n}$ and
$$ \pi: \Gamma^{n} \to \Gamma, \ \ \ (x_{1},\ldots,x_{n}) \mapsto
   x_{1} $$
be the projection onto the first factor. Suppose that $\infty$ is
an accumulation point of $\pi(P)$. Then there is an affine
semi-line
$$ L = \{ (r_{1}t + \gamma_{1},\ldots,r_{n}t + \gamma_{n}): \, t
   \in \Gamma, \ t \geq 0 \} \ \ \text{with} \ \ r_{1},\ldots,r_{n}
   \in \mathbb{N}, \, r_{1} >0, $$
passing through a point $\gamma = (\gamma_{1},\ldots,\gamma_{n})
\in P$ and such that $\infty$ is an accumulation point of $\pi(P
\cap L)$ too.
\end{lemma}

The above two lemmas will be often used in further reasonings.

\vspace{1ex}

As for the theorem on existence of the limit, here we first prove,
using Theorem~\ref{term}, a weaker version given below. The full
analytic version (Proposition~\ref{limit-th2}) will be established
in Section~4 by means of the closedness theorem.

\begin{proposition}\label{limit-th1}
Let $f:E \to K$ be an $\mathcal{L}$-definable function on a subset
$E$ of $K$ and suppose $0$ is an accumulation point of $E$. Then
there is an $\mathcal{L}$-definable subsets $F \subset E$ with
accumulation point $0$ and a point $w \in \mathbb{P}^{1}(K)$ such
that
$$ \lim_{x \rightarrow 0}\, f|F\, (x) = w. $$
Moreover, we can require that
$$ \{ (x,f(x)): x \in F \} \subset \{ (x^{r}, \phi(x)): x \in G
   \}, $$
where $r$ is a positive integer and $\phi$ is a definable
function, a composite of some functions induced by series from
$\mathcal{A}$ and of some algebraic power series (coming from the
implicit function theorem). Then, in particular, the definable set
$$ \{ (v(x), v(f(x))): \; x \in (F \setminus \{0 \} \}
   \subset \Gamma \times (\Gamma \cup \{\infty \}) $$
is contained in an affine line with rational slope
$$ q \cdot l = p \cdot k + \beta,  $$
with $p,q \in \mathbb{Z}$, $q>0$, $\beta \in \Gamma$, or in\/
$\Gamma \times \{ \infty \}$.
\end{proposition}

\begin{proof}
In view of Remark~\ref{ext-par}, we may assume that $K$ has
separated analytic $\mathcal{A}(K)$-structure. We apply
Theorem~\ref{term} and proceed with induction with respect to the
complexity of the term $t$. Since an angular component map exists
(cf.~Remark~\ref{angular}), the sorts $\widetilde{K}$ and $\Gamma$
are orthogonal in
$$ RV(K) \simeq \widetilde{K} \times \Gamma_{K}. $$
Therefore, after shrinking $F$, we can assume that
$\overline{ac}\: (F) = \{1 \}$ and the function $g$ goes into $ \{
\xi \} \times \Gamma^{s} $ with a $\xi \in \widetilde{K}^{s}$, and
next that $\xi = (1,\ldots,1)$; similar reductions were considered
in our papers~\cite{Now-Sel,Now-thm}. For simplicity, we look at
$g$ as a function into $\Gamma^{s}$. We shall briefly explain the
most difficult case where
$$ t(x,g(x)) = h_{m}(a_{0}(x),\ldots,a_{m}(x),(1,g_{0}(x))), $$
assuming that the theorem holds for the terms
$a_{0},\ldots,a_{m}$; here $g_{0}$ is one of the components of
$g$. In particular, each function  $a_{i}(x)$ has, after
partitioning, a limit, say, $a_{i}(0)$ when $x$ tends to zero.

\vspace{1ex}

By Lemma~\ref{line-2}, we can assume that
\begin{equation}\label{eq-limit-1}
p v(x) + q g_{0}(x) + v(a) = 0
\end{equation}
for some $p,q \in \mathbb{N}$ and $a \in K \setminus \{ 0 \}$. By
the induction hypothesis, we get
$$ \{ (x,a_{i}(x)): x \in F \} \subset \{ (x^{r}, \alpha_{i}(x)): x \in
   G \}, \ \ \ i = 0,1,\ldots,m, $$
for some power series $\alpha_{i}(x))$ as stated in the theorem.
Put
$$ P(x,T) := \sum_{i=0}^{m} \, a_{i}(x) T^{i}. $$
By the very definition of $h_{m}$ and since we are interested in
the vicinity of zero, we may assume that there is an
$i_{0}=0,\ldots,m$ such that
$$ \forall \: x \in F \ \exists \: u \in K \ \
   v(u) = g_{0}(x), \ \ \overline{ac} \: u =1, $$
and the following formulas hold
\begin{equation}\label{eq-limit-2}
  v(a_{i_{0}}(x)u^{i_{0}}) = \min \, \{ v(a_{i}(x)u^{i}), \
i=0,\ldots,m \},
\end{equation}
$$ v(P(x,u)) > v(a_{i_{0}}(x)u^{i_{0}}), \ \ v \left( \frac{\partial
   \, P}{\partial \, T} (x,u) \right) = v(a_{i_{0}}(x)u^{i_{0}-1}). $$
Then $h_{m}(a_{0}(x),\ldots,a_{m}(x),(1,g_{0}(x)))$ is a unique
$b(x) \in K$ such that
$$ P(x,b(x))=0, \ \ v(b(x)) = g_{0}(x), \ \ \overline{ac} \: b(x)
   =1. $$
By \cite[Remarks~7.2, 7.3]{Now-thm}, the set $F$ contains the set
of points of the form $c^{r}t^{Nqr}$ for some $c \in K$ with
$\overline{ac} \: c=1$, a positive integer $N$ and all $t \in
K^{\circ}$ small enough with $\overline{ac} \: t =1$. Hence and by
equation~\ref{eq-limit-1}, we get
$$ g_{0}(c^{r}t^{Nqr}) = g_{0}(c^{r}) - v(t^{Npr}). $$
Take $d \in K$ such that $g_{0}(c^{r}) =v(d)$ and $\overline{ac}
\: d=1$. Then
$$ g_{0}(c^{r}t^{Nqr}) = v(dt^{-Npr}). $$
Thus the homothetic change of variable
$$ Z = T/dt^{-Npr} = t^{Npr}T/d $$
transforms the polynomial
$$ P(c^{r}t^{Nqr},T) = \sum_{i=0}^{m} \, \alpha_{i}(ct^{Nq}) T^{i} $$
into a polynomial $Q(t,Z)$ to which Hensel's lemma applies
(cf.~\cite[Lemma~3.5]{Pa1}):

\begin{equation}
  P(c^{r}t^{Nqr},T) = P(c^{r}t^{Nqr},dt^{-Npr}Z) =
  \end{equation}
$$ \sum_{i=0}^{m} \, \alpha_{i}(ct^{Nq}) \cdot ( dt^{-Npr} Z)^{i} =
   (\alpha_{i_{0}}(ct^{Nq}) \cdot (dt^{-Npr})^{i_{0}} \cdot Q(t,Z).
$$
Indeed, formulas~\ref{eq-limit-2} imply that the coefficients
$b_{i}(t)$, $i=0,\ldots,m$, of the polynomial $Q$ are power series
(of order $\geq 0$) in the variable $t$, and that
$$ v(Q(t,1)) > 0 \ \ \ \text{and} \ \ \ v\left( \frac{\partial \,
   Q}{\partial \, Z} (t,1) \right) = 0 $$
fot $t \in K^{0}$ small enough. Hence
$$ v(Q(0,1)) > 0 \ \ \ \text{and} \ \ \ v\left( \frac{\partial \,
   Q}{\partial \, Z} (0,1) \right) = 0. $$
But, for $x(t) = c^{r}t^{Nqr}$, the unique zero $T(t) = b(x(t))$
of the polynomial $P(x(t),T)$ such that
$$ v(b(x(t))) = v(dt^{-Npr}) \ \ \text{and} \ \ \overline{ac}\: b(x(t)) =1 $$
corresponds to a unique zero $Z(t)$ of the polynomial $Q(t,Z)$
such that
$$ v(Z(t)) = v(1) \ \ \text{and} \ \ \overline{ac}\: Z(t) =1. $$
Therefore the conclusion of the theorem can be directly obtained
via the implicit function theorem
(cf.~\cite[Proposition~~2.5]{Now-thm}) applied to the polynomial
$$ P(A_{0},\ldots,A_{m},Z) = \sum_{i=0}^{m} \, A_{i} Z^{i} $$
in the variables $A_{i}$ substituted for $a_{i}(x)$ at the point
$$ A_{0}=b_{0}(0),\, \ldots,\, A_{m}=b_{m}(0),\, Z=1. $$
\end{proof}


Now we can readily proceed with the

\vspace{1ex}

{\em Proof of the closedness theorem} (Theorem~\ref{clo-th}). We
must show that if $B$ is an $\mathcal{L}$-definable subset of $D
\times (K^{\circ})^{n}$ and a point $a$ lies in the closure of $A
:= \pi(B)$, then there is a point $b$ in the closure of $B$ such
that $\pi(b)=a$. As before (cf.~\cite[Section~~8]{Now-thm}), the
theorem reduces easily to the case $m=1$ and next, by means of
fiber shrinking (Proposition~\ref{FS}), to the case $n=1$. We may
obviously assume that $a = 0 \not \in A$.

\vspace{1ex}

By b-minimal cell decomposition, we can assume that the set $B$ is
a relative cell with center over $A$. It means that $B$ has a
presentation of the form
$$ \Lambda: B \ni (x,y) \to (x,\lambda(x,y)) \in  A \times RV(K)^{s}, $$
where $\lambda: B \to RV(K)^{s}$ is an $\mathcal{L}$-definable
function, such that for each $(x,\xi) \in \Lambda (B)$ the
pre-image $\lambda_{x}^{-1}(\xi) \subset K$ is either a point or
an open ball; here $\lambda_{x}(y) := \lambda(x,y)$. In the latter
case, there is a center, i.e.\ an $\mathcal{L}$-definable map
$\zeta: \Lambda(B) \to K$, and a (unique) map $\rho: \Lambda (B)
\to RV(K) \setminus \{ 0 \}$ such that
$$ \lambda_{x}^{-1}(\xi) = \{ y \in K: rv\, (y - \zeta (x,\xi)) =
   \rho (x,\xi) \} . $$
Again, since the sorts $\widetilde{K}$ and $\Gamma$ are orthogonal
in $RV(K) \simeq \widetilde{K} \times \Gamma$, we can assume,
after shrinking the sets $A$ and $B$, that
$$ \lambda(B) \subset \{ (1,\ldots,1) \} \times \Gamma^{s} \subset
   \widetilde{K}^{s} \times \Gamma^{s}; $$
let $\tilde{\lambda}(x,y)$ be the projection of $\lambda(x,y)$
onto $\Gamma^{s}$. By Lemma~\ref{line-2}, we can assume once
again, after shrinking the sets $A$ and $B$, that the set
$$ \{ (v(x),v(y),\tilde{\lambda}(x,y)): \; (x,y) \in B \} \subset
   \Gamma^{s+2} $$
is contained in an affine semi-line with integer coefficients.
Hence $\lambda(x,y) = \phi(v(x))$ is a function of one variable
$x$. We have two cases.

\vspace{1ex}

{\em Case I.} $\lambda_{x}^{-1}(\xi) \subset K^{\circ}$ is a
point. Since each $\lambda_{x}$ is a constant function, $B$ is the
graph of an $\mathcal{L}$-definable function. The conclusion of
the theorem follows thus from Proposition~\ref{limit-th1}.

\vspace{1ex}

{\em Case II.} $\lambda_{x}^{-1}(\xi) \subset K^{\circ}$ is a
ball. Again, application of Lemma~\ref{line-2} makes it possible,
after shrinking the sets $A$ and $B$, to arrange the center
$$ \zeta: \Lambda(B) \ni (x, \xi) \to \zeta(x, v(x)) = \zeta(x) \in K $$
and the function $\rho(x, \xi) = \rho(v(x))$ as functions of one
variable $x$. Likewise as it was above, we can assume that the set
$$ P := \{ (v(x), \rho(v(x))) : x \in A \} \subset \Gamma^{2} $$
is contained in an affine line $p v(x) +q \rho(v(x)) + v(c) =0$
with integer coefficients $p,q$, $q \neq 0$; furthermore, that $P$
contains the set
$$ Q := \{ (v(ct^{qN}), \rho(v(ct^{qN}))): \; t \in K^{\circ} \}
$$
for a positive integer $N$. Then we easily get
$$ \rho(v(ct^{qN})) = \rho(c) - pN v(t) = v(ct^{-pN}). $$
Hence the set $B$ contains the graph
$$ \{ (ct^{qN}, \zeta(ct^{qN}) + ct^{-pN}) : \; t \in K^{\circ} \}. $$
As before, the conclusion of the theorem follows thus from
Proposition~\ref{limit-th1}, and the proof is complete.
\hspace{\fill} $\Box$

\vspace{1ex}

\section{Direct applications}

The framework of b-minimal structures provides cell decomposition
and a good concept of dimension (cf.~\cite{C-L}), which in
particular satisfies the axioms from the paper~\cite{Dries-dim}.
For separated analytic structures, the zero-dimensional sets are
precisely the finite sets, and also valid is the following
dimension inequality, which is of great geometric significance:
\begin{equation}\label{ineq}
\dim \, \partial E < \dim E;
\end{equation}
here $E$ is any $\mathcal{L}$-definable subset of $K^{n}$ and
$\partial E := \overline{E} \setminus E$ denotes the frontier of
$A$.

\vspace{1ex}

We first apply the closedness theorem to obtain the following full
analytic version of the theorem on existence of the limit.

\begin{proposition}\label{limit-th2}
Let $f: E \to \mathbb{P}^{1}(K)$ be an $\mathcal{L}$-definable
function on a subset $E$ of $K$, and suppose that $0$ is an
accumulation point of $E$. Then there is a finite partition of $E$
into $\mathcal{L}$-definable sets $E_{1},\ldots,E_{r}$ and points
$w_{1}\ldots,w_{r} \in \mathbb{P}^{1}(K)$ such that
$$ \lim_{x \rightarrow 0}\, f|E_{i}\, (x) = w_{i} \ \ \ \text{for} \ \
   i=1,\ldots,r. $$
\end{proposition}

\begin{proof}

We may of course assume that $0 \not \in E$. Put
$$ F := \mathrm{graph}\, (f) = \{ (x, f(x): x \in E \} \subset K
   \times \mathbb{P}^{1}(K); $$
obviously, $F$ is of dimension $1$. It follows from the closedness
theorem that the frontier $\partial F \subset K \times
\mathbb{P}^{1}(K)$ is non-empty, and thus of dimension zero by
inequality~\ref{ineq}. Say
$$ \partial F \cap (\{ 0 \} \times \mathbb{P}^{1}(K)) = \{
   (0,w_{1}), \ldots, (0,w_{r}) \} $$
for some $w_{1},\ldots,w_{r} \in \mathbb{P}^{1}(K).$ Take pairwise
disjoint neighborhoods $U_{i}$ of the points $w_{i}$,
$i=1,\ldots,r$, and set
$$ F_{0} := F \cap \left(E \times \left( \mathbb{P}^{1}(K) \setminus \bigcup_{i}^{r}
   E_{i} \right) \right). $$
Let
$$ \pi: K \times \mathbb{P}^{1}(K) \longrightarrow K $$
be the canonical projection. Then
$$ E_{0} := \pi (F_{0}) = f^{-1}\left( \mathbb{P}^{1}(K) \setminus \bigcup_{i}^{r}
   E_{i} \right). $$
Clearly, the closure $\overline{F}_{0}$ of $F_{0}$ in $K \times
\mathbb{P}^{1}(K))$ and $\{ 0 \} \times \mathbb{P}^{1}(K))$ are
disjoint. Hence and by the closedness theorem, $0 \not \in
\overline{E_{0}}$, the closure of $E_{0}$ in $K$. The set $E_{0}$
is thus irrelevant with respect to the limit at $0 \in K$.
Therefore it remains to show that
$$ \lim_{x \rightarrow 0}\, f|E_{i}\, (x) = w_{i} \ \ \ \text{for} \ \
   i=1,\ldots,r. $$
Otherwise there is a neighborhood $V_{i} \subset U_{i}$ such that
$0$ would be an accumulation point of the set
$$ f^{-1}(U_{i} \setminus V_{i}) = \pi (F \cap (E
   \times (U_{i} \setminus V_{i}))). $$
Again, it follows from the closedness theorem that $\{ 0 \} \times
\mathbb{P}^{1}(K)$ and the closure of $F \cap (E \times (U_{i}
\setminus V_{i}))$ in $K \times \mathbb{P}^{1}(K))$ would not be
disjoint. This contradiction finishes the proof.
\end{proof}

\begin{remark}
Let us mention that Proposition~\ref{limit-th2} can be
strengthened as stated below (cf.\ the algebraic versions
\cite[Proposition~5.2]{Now-Sel} and \cite[Theorem~5.1]{Now-thm}):

\vspace{1ex}

\begin{em}
Moreover, perhaps after refining the finite partition of $E$,
there is a neighbourhood $U$ of $0$ such that each definable set
$$ \{ (v(x), v(f(x))): \; x \in (E_{i} \cap U) \setminus \{0 \} \}
   \subset \Gamma \times (\Gamma \cup \ \{
   \infty \}),  \ i=1,\ldots,r, $$
is contained in an affine line with rational slope
$$ q \cdot l = p_{i} \cdot k + \beta_{i}, \ i=1,\ldots,r, $$
with $p_{i},q \in \mathbb{Z}$, $q>0$, $\beta_{i} \in \Gamma$, or
in\/ $\Gamma \times \{ \infty \}$.
\end{em}
\end{remark}

Now we turn to a second application, namely the following theorem
on piecewise continuity.

\begin{theorem}\label{piece}
Let $A \subset K^{n}$ and $f: A \to \mathbb{P}^{1}(K)$ be an
$\mathcal{L}$-definable function. Then $f$ is piecewise
continuous, i.e.\ there is a finite partition of $A$ into
$\mathcal{L}$-definable locally closed subsets
$A_{1},\ldots,A_{s}$ of $K^{n}$ such that the restriction of $f$
to each $A_{i}$ is continuous.
\end{theorem}

\begin{proof}
Consider an $\mathcal{L}$-definable function $f: A \to
\mathbb{P}^{1}(K)$ and its graph
$$ E := \{ (x,f(x)): x \in A \} \subset K^{n} \times \mathbb{P}^{1}(K). $$
We shall proceed with induction with respect to the dimension
$$ d = \dim A = \dim \, E $$
of the source and graph of $f$.

\vspace{1ex}

Observe first that every $\mathcal{L}$-definable subset $E$ of
$K^{n}$ is a finite disjoint union of locally closed
$\mathcal{L}$-definable subsets of $K^{n}$. This can be easily
proven by induction on the dimension of $E$ by means of
inequality~\ref{ineq}. Therefore we can assume that the graph $E$
is a locally closed subset of $K^{n} \times \mathbb{P}^{1}(K)$ of
dimension $d$ and that the conclusion of the theorem holds for
functions with source and graph of dimension $< d$.

\vspace{1ex}

Let $F$ be the closure of $E$ in $K^{n} \times \mathbb{P}^{1}(K)$
and $\partial E := F \setminus E$ be the frontier of $E$. Since
$E$ is locally closed, the frontier $\partial E$ is a closed
subset of $K^{n} \times \mathbb{P}^{1}(K)$ as well. Let
$$ \pi: K^{n} \times \mathbb{P}^{1}(K) \longrightarrow K^{n} $$
be the canonical projection. Then, by virtue of the closedness
theorem, the images $\pi(F)$ and $\pi(\partial E)$ are closed
subsets of $K^{n}$. Further,
$$ \dim \, F = \dim \, \pi(F) = d $$
and
$$ \dim \, \pi(\partial E) \leq \dim \, \partial E < d; $$
the last inequality holds by inequality~\ref{ineq}. Putting
$$ B := \pi(F) \setminus \pi(\partial E) \subset \pi(E) = A, $$
we thus get
$$ \dim \, B = d \ \ \text{and} \ \ \dim \, (A \setminus B) < d.
$$
Clearly, the set
$$ E_{0} := E \cap (B \times \mathbb{P}^{1}(K)) = F \cap (B \times
   \mathbb{P}^{1}(K)) $$
is a closed subset of $B \times \mathbb{P}^{1}(K)$ and is the
graph of the restriction
$$ f_{0}: B \longrightarrow \mathbb{P}^{1}(K) $$
of $f$ to $B$. Again, it follows immediately from the closedness
theorem that the restriction
$$ \pi_{0} : E_{0} \longrightarrow B $$
of the projection $\pi$ to $E_{0}$ is a definably closed map.
Therefore $f_{0}$ is a continuous function. But, by the induction
hypothesis, the restriction of $f$ to $A \setminus B$ satisfies
the conclusion of the theorem, whence so does the function $f$.
This completes the proof.
\end{proof}

We immediately obtain

\begin{corollary}
The conclusion of the above theorem holds for any
$\mathcal{L}$-definable function $f: A \to K$.
\end{corollary}

\vspace{1ex}

\section{The \L{}ojasiewicz inequalities}

Algebraic non-Archimedean versions of the \L{}ojasiewicz
inequality, established in our papers~\cite{Now-Sel,Now-thm}, can
be carried over to the analytic settings considered here with
proofs repeated almost verbatim. We thus state only the results
(Theorems 11.2, 11.5 and 11.6, Proposition~11.3 and Corollary~11.4
from~\cite{Now-thm}). Let us mention that the main ingredients of
the proof are the closedness theorem, elimination of valued field
quantifiers, the orthogonality of the auxiliary sorts and relative
quantifier elimination for ordered abelian groups. They allow us
to reduce the problem under study to that of piecewise linear
geometry. We first state the following version, which is closest
to the classical one.

\begin{theorem}\label{Loj1}
Let $f,g_{1},\ldots,g_{m}: A \to K$ be continuous
$\mathcal{L}$-definable functions on a closed (in the
$K$-topology) bounded subset $A$ of $K^{m}$. If
$$ \{ x \in A: g_{1}(x)= \ldots =g_{m}(x) =0 \} \subset \{ x \in A: f(x)=0 \}, $$
then there exist a positive integer $s$ and a constant $\beta \in
\Gamma$ such that
$$ s \cdot v(f(x)) + \beta \geq v((g_{1}(x), \ldots ,g_{m}(x))) $$
for all $x \in A$. Equivalently, there is a $C \in |K|$ such that
$$ | f(x) |^{s} \leq C \cdot \max \, \{ | g_{1}(x) |, \ldots , |
   g_{m}(x) | \} $$
for all $x \in A$.
\end{theorem}

A direct consequence of Theorem~\ref{Loj1} is the following result
on H\"{o}lder continuity of definable functions.

\begin{proposition}\label{Hol}
Let $f: A \to K$ be a continuous $\mathcal{L}$-definable function
on a closed bounded subset $A \subset K^{n}$. Then $f$ is
H\"{o}lder continuous with a positive integer $s$ and a constant
$\beta \in \Gamma$, i.e.\
$$ s \cdot v(f(x) - f(z)) + \beta \geq  v(x-z) $$
for all $x,z \in A$. Equivalently, there is a $C \in |K|$ such
that
$$ | f(x) - f(z) |^{s} \leq C \cdot | x-z | $$
for all $x,z \in A$.
\end{proposition}

We immediately obtain

\begin{corollary}
Every continuous $\mathcal{L}$-definable function $f: A \to K$ on
a closed bounded subset $A \subset K^{n}$ is uniformly continuous.
\end{corollary}

Now we formulate another, more general version of the
\L{}ojasiewicz inequality for continuous definable functions of a
locally closed subset of $K^{n}$.

\begin{theorem}\label{Loj2}
Let $f,g: A \to K$ be two continuous $\mathcal{L}$-definable
functions on a locally closed subset $A$ of $K^{n}$. If
$$ \{ x \in A: g(x)=0 \} \subset \{ x \in A: f(x)=0 \}, $$
then there exist a positive integer $s$ and a continuous
$\mathcal{L}$-definable function $h$ on $A$ such that $f^{s}(x) =
h(x) \cdot g(x)$ for all $x \in A$.
\end{theorem}

Finally, put
$$ \mathcal{D}(f) := \{ x \in A: f(x)
   \neq 0 \} \ \ \text{and} \ \ \mathcal{Z}\, (f) := \{ x \in A: f(x) = 0 \}. $$
The following theorem may be also regarded as a kind of the
\L{}ojasiewicz inequality, which is, of course, a strengthening of
Theorem~\ref{Loj2}.

\begin{theorem}\label{Loj3}
Let $f: A \to K$ be a continuous $\mathcal{L}$-definable function
on a locally closed subset $A$ of $K^{n}$ and $g: \mathcal{D}(f)
\to K$ a continuous $\mathcal{L}$-definable function. Then $f^{s}
\cdot g$ extends, for $s \gg 0$, by zero through the set
$\mathcal{Z}\, (f)$ to a (unique) continuous
$\mathcal{L}$-definable function on $A$.
\end{theorem}

\vspace{1ex}

\section{Curve selection}

Consider a Henselian field $K$ with a separated analytic
$\mathcal{A}$-structure. In this section, we establish a general
version of curve selection for $\mathcal{L}$-definable sets. Note
that the domain of the selected curve is, unlike in the classical
version, only an $\mathcal{L}$-definable subset of the unit disk.

\begin{proposition}\label{GCSL}
Let $A$ be an $\mathcal{L}$-definable subset of $K^{p}$. If a
point $a \in K^{p}$ lies in the closure (in the $K$-topology)
$\mathrm{cl}\, (A \setminus \{ a \})$ of $A \setminus \{ a \}$,
then there exist an $\mathcal{L}$-definable map $\varphi :
K^{\circ} \longrightarrow K^{p}$ given by power series from
$A_{p}^{loc}(K)$, and an $\mathcal{L}$-definable subset $E$ of
$K^{\circ}$ with accumulation point $0$ such that
$$ \varphi(0)=a \ \ \ {\text and} \ \ \  \varphi( E
   \setminus \{ 0 \}) \subset A \setminus \{ a \}. $$
\end{proposition}

\begin{proof}
We call the problem under study curve selection for the couple
$(A,a)$. We may assume without loss of generality that $a=0 \in
K^{p}$. By elimination of valued field quantifiers, the set $A
\setminus \{ a \}$ is a finite union of sets defined by conditions
of the form
$$ (v(t_{1}(x)),\ldots,v(t_{r}(x))) \in P, \ \
   (\overline{ac}\, \tau_{1}(x),\ldots,\overline{ac}\, \tau_{s}(x)) \in Q,
$$
where $t_{i},\tau_{j}$ are terms of the separated analytic
structure $\mathcal{A}(K)$, and $P$ and $Q$ are definable subsets
of $\Gamma^{r}$ and $\Bbbk^{s}$, respectively. Here it is
convenient to deal with a local concept of term, i.e.\ a finite
composite of functions analytic near a given point (in some local
analytic coordinates) and the reciprocal function $1/x$.

\vspace{1ex}

One can, of course, assume that $A$ is just a set of this form. We
shall proceed with induction on the complexity of these terms. Its
lowering is possible via successive transformations to a normal
crossing by means of Corollary~\ref{cor-NC} and the three
straightforward observations below.

\vspace{2ex}

{\bf Observation 1.}
Consider a finite composite of blow-ups
$$ \pi_{0}: \widetilde{X}_{0} \to \Delta_{p}(r) $$
from Corollary~\ref{cor-NC} and put $B := \pi_{0}^{-1}(A \setminus
\{ a \})$. Since $\pi_{0}$ is a surjective, definably closed map
by Corollary~\ref{clo-th-cor-0} to the closedness theorem, there
is a point $b \in \mathrm{cl}\, (B) \setminus B$ such that
$\pi_{0}(b) =a$. Clearly, if the couple $(B,b)$ satisfies the
conclusion of Proposition~\ref{GCSL}, so does the couple $(A,a)$.

\vspace{2ex}

{\bf Observation 2.}
Suppose that a finite number of $\mathcal{L}$-terms $t_{i}$,
$i=1,\ldots,l$, have been already transformed to normal crossing
divisors with respect to some local analytic coordinates
$x=(x_{1},\ldots,x_{p})$ near a point $a$. Next consider a finite
number of other functions $f_{j}$, $j=1,\ldots,q$, analytic near
$a$. After simultaneous transformation of the functions $f_{j}$
and the coordinates $x_{k}$ to a normal crossing (possibly taking
into account divisibility relation), all the terms $t_{i}$ and
functions $f_{j}$ become normal crossing divisors (along the fiber
over the point $a$).

\vspace{2ex}

{\bf Observation 3.} Let $t(x)$ be a term of the form
$$ h(f_{1}/g_{1}(x),\ldots,f_{k}/g_{k}(x)), $$
where $f_{i},g_{i}$ are analytic functions near a point $a = 0 \in
K^{p}$ and $h$ is the interpretation of a function of the language
$\mathcal{L}$ (extended by zero off its natural domain). Consider
simultaneous transformation to a normal crossing of the functions
$f_{i}$, $g_{i}$ which takes into account linear ordering with
respect to divisibility relation and a point $b$ such that
$\pi_{0}(b)=a$. Then we can assume without loss of generality that
the quotients $f_{i}/g_{i} \circ \pi_{0}$ are normal crossing
divisors at $b$. Otherwise the term $t \circ \pi_{0}$ would vanish
near $b$, and then we would pass to a term of lower complexity.

\vspace{2ex}

Now it is not difficult to reduce curve selection for the initial
couple $(A,a)$ to curve selection for a couple $(B,b)$ where $B$
is a set defined by conditions of the form
$$ (v(t_{1}(y)),\ldots,v(t_{r}(y))) \in P, \ \
   (\overline{ac}\, \tau_{1}(y),\ldots,\overline{ac}\, \tau_{s}(y)) \in Q,
$$
where $y$ are suitable local analytic coordinates near $b$, each
of the $t_{i},\tau_{j}$ is either a normal crossing $u(y) \cdot
y^{\alpha}$ at $b$, or a reciprocal normal crossing $u(y) \cdot
1/y^{\alpha}$ at $b$, where $u(b) \neq 0$ and $\alpha \in
\mathbb{N}^{p}$, or vanishes near $b$.

\vspace{1ex}

Since the valuation map and the angular component map composed
with a continuous function are locally constant near any point at
which this function does not vanish, the conditions which describe
the set $B$ near $b$ can be easily expressed in the form
$$ (v(y_{1}),\ldots,v(y_{p})) \in \widetilde{P}, \ \
   (\overline{ac}\, y_{1},\ldots,\overline{ac}\, y_{p}) \in
   \widetilde{Q}, $$
where $\widetilde{P}$ and $\widetilde{Q}$ are definable subsets of
$\Gamma^{p}$ and $\widetilde{K}^{p}$, respectively.

\vspace{1ex}

We thus achieved the same reduction as in the algebraic case
studied in our papers~\cite{Now-Sel,Now-thm}. In this manner, we
can repeat verbatim the remaining part of the proof given in those
papers. The main ingredient of the further reasoning is
Lemma~\ref{line-1} (\cite[Lemma~6.2]{Now-thm}) which, in turn,
relies on relative quantifier elimination for ordered abelian
groups.
\end{proof}

We conclude the paper with the following comment.

\begin{remark}
In the recent paper~\cite{Now-TU}, we proved that every closed
definable subset of $K^{n}$ is a definable retract of $K^{n}$.
Consequently, we established a non-Archimedean version of the
Tietze--Urysohn theorem on extending continuous functions
definable over Henselian valued fields of equicharacteristic zero.
It is very plausible that these results will also hold over
Henselian valued fields with analytic structure.
\end{remark}

\vspace{3ex}

\begin{small}
Institute of Mathematics

Faculty of Mathematics and Computer Science

Jagiellonian University


ul.~Profesora S.\ \L{}ojasiewicza 6

30-348 Krak\'{o}w, Poland

{\em E-mail address: nowak@im.uj.edu.pl}
\end{small}

\end{document}